\documentclass[12pt,a4paper,reqno]{amsart}
\usepackage[latin1]{inputenc}
\usepackage{amsmath}
\usepackage{amsfonts}
\usepackage{amssymb}
\usepackage{amsmath}
\usepackage{amsthm}
\usepackage{graphicx}

\usepackage[margin=3cm]{geometry}
\newtheorem{theorem}{Theorem}
\newtheorem{lemma}{Lemma}
\newtheorem{cor}{Corollary}
\newtheorem{conj}{Conjecture}

\newtheorem{remark}{Remark}
\newtheorem{defn}{Definition}
\newtheorem{example}{Example}

\author{Ben Lawrence}
\address{Department of Mathematics, University of Auckland}
\email{ben.lawrence@auckland.ac.nz}
\thanks{Supported by the Marsden Fund Council of the Royal Society of New Zealand.}

\title[Burnside graphs and the Kippenhahn Conjecture]{Burnside graphs, algebras generated by sets of matrices, and the Kippenhahn Conjecture}
\date{\today}
\begin{document}

\begin{abstract}
Given a set of matrices, it is often of interest to determine the algebra they generate. Here we exploit the concept of the Burnside graph of a set of matrices, and show how it may be used to deduce properties of the algebra they generate. We prove two conditions regarding a set of matrices generating the full algebra; the first necessary, the second sufficient. An application of these results is given in the form of a new family of counterexamples to the Kippenhahn conjecture, of order $8 \times 8$ and greater.
\end{abstract}

\maketitle

\section{Introduction}\label{intro}

One of our main goals is to determine whether or not a set of $n \times n$ matrices over a field $F$ generates the full matrix algebra $M_{n}(F)$. Various authors have looked at this problem. Here is a small sample of the vast literature on this question. Kostov \cite{min_gens} placed minimum bounds on the number of complex matrices required to generate a subalgebra of $M_{n}(\mathbb{C})$. For the case of two matrices, one of which has distinct eigenvalues, George and Ikramov gave a criterion for when they cannot generate the full algebra \cite{common_invariant}. Again using the assumption that one of the generating matrices has distinct eigenvalues, Laffey gave two separate criteria for generation of the full algebra \cite{laffey_struct,laffey_full_gen}. Aslaksen and Sletsj{\o}e in their 2009 paper \cite{generators} published some criteria for $ n = 2 \mbox{ or } 3$. We will go beyond the distinct eigenvalue requirement, and give criteria for the case of repeated eigenvalues.

In accordance with Burnside's theorem for matrix algebras, given for instance as Corollary 5.23 of Bre\v{s}ar \cite{burnside_proof}, a set of $n \times n$ complex matrices generates the full algebra $M_{n}(\mathbb{C})$ if and only if they have no invariant subspaces in common. In the appopriate basis, these invariant subspaces are immediately apparent. We will now define the Burnside graph of a set of matrices to help the invariant subspaces emerge. This definition is adapted from \cite{laffey_struct}. \begin{defn}[Burnside graph]\label{burnside-1}
Let $A = \left\lbrace A_{1},...,A_{k}\right\rbrace$ be a set of $n \times n$ matrices over a field $F$. The \textbf{Burnside graph} $B(A)$ of $A$ is a directed graph of $n$ nodes $\left\lbrace 1,...,n \right\rbrace$, with a directed edge existing from node $i$ to node $j$ if and only if there is some matrix $A_{m}$ with a non-zero entry at the $(i,j)$ position. Self-loops are not considered. A Burnside graph $B(A)$ is \textbf{strongly connected} if every pair of nodes in $B(A)$ is path-connected in both directions.

The Burnside graph $B(A)$ as defined above is formed by treating each $A_{i}$ as an adjacency matrix without regard to weighting, constructing all of the associated graphs, and merging them all.

\end{defn} 

See section $ \ref{obstacles}$  for an example of a Burnside graph.  When the graph has certain non-connectivity properties, it is guaranteed that the set of matrices does not generate the full algebra. However, the Burnside graph will change if the basis is changed. It is necessary in a sense to be in the correct basis in order to see the invariant subspaces. This brings us to our first theorem, which connects the algebra generated by a tuple of matrices $A$ with the strong connectedness of the Burnside graph $B(A)$. Since a tuple of $n \times n$ matrices over a field $F$ generates $M_{n}(F)$ if and only if they generate over the algebraic closure $\overline{F}$ of $F$ the full matrix algebra $M_{n}(\overline{F})$, we will focus mainly on algebraically closed fields of characteristic 0, and $\mathbb{C}$ in particular. We will occasionally also work over $\mathbb{R}$.

\begin{theorem}[Obstacle to full algebra]\label{not_gen_full_thm}
Let $A = \left\lbrace A_{1},...,A_{k}\right\rbrace$ be a set of $n \times n$ matrices over $\mathbb{C}$, and let $\mathcal{A}$ be the algebra generated by $A$. If $B(A)$ is not strongly connected, then $\mathcal{A} \neq M_{n}(\mathbb{C})$.
\end{theorem} The proof of this theorem will be given in section $ \ref{obstacles} $.

This leads to our main theorem. Here we present a simplified version, to give an impression of the full version which can be found along with its proof in Section \ref{algebras}.

\begin{theorem}[Even-order constructibility - special case]\label{2_gens_thm_simple}
Let $H$ and $K$ be $2n \times 2n$ hermitian matrices over $\mathbb{C}$. If \begin{enumerate}

\item $K$ is diagonal with eigenvalues each of multiplicity 2,
\item $B(H, K)$ is strongly connected,
\item the top row of $2 \times 2$ blocks of $H$ are all invertible,
\item there exist distinct top-row $2 \times 2$ blocks $H_{1j}$ and $H_{1k}$ of $H$ such that $H_{1j}H_{1j}^{T}$ and $H_{1k}H_{1k}^{T}$ do not commute,
\end{enumerate} then the algebra $\mathcal{A}$ generated by $H$ and $K$ is the full algebra $M_{2n}(\mathbb{C})$. \\
\end{theorem} 

Since the full algebra can always be 2-generated \cite{min_gens}, it is of particular interest to tell whether or not a given pair of matrices generates the full algebra. In the case of a pair of hermitian matrices, we can diagonalise one of them, and find a basis within which to assess the Burnside graph. The Burnside graph illustrates the decomposition, according to Burnside's theorem, of a set of matrices into blocks corresponding to invariant subspaces. 

Our work was motivated by the 1951 conjecture of Rudolf Kippenhahn \cite{kip}: 

\begin{conj}[Kippenhahn \cite{kip}] \label{kip}
Let and $H, K$ be hermitian $2n \times 2n$ matrices, and let $f = \det(xH + yK + I) \in \mathbb{R}[x,y]$. Let $\mathcal{A}$ be the algebra generated by $H$ and $K$. If there exists $ g \in \mathbb{C}[x,y]$ such that $ f = g^{k}$ then there is some unitary matrix $U$ such that $U^{*}(xH + yK)U$ is block diagonal, and thus $ \mathcal{A} \neq M_{n}(\mathbb{C})$.
\end{conj} Our goal has been to further understanding of the counterexamples to this conjecture. Kippenhahn orginally gave a more general form of this conjecture where $f$ is permitted to be a product of more than one irreducible polynomial. Kippenhahn's conjecture linked the multiplicity of eigenvalues of a certain type of matrix polynomial to the algebra generated by the coefficients of that polynomial. The claim was that, given hermitian $H$ and $K$, if the polynomial $H x + K y$ for scalar $x$ and $y$ has, for all $x$ and $y$, eigenspaces each of even dimension, then $H$ and $K$ cannot generate the full algebra. More intuitively, if $H x + K y$ had square characteristic polynomial $f^{2}$, then the conjecture was that one could transform $H$ and $K$ simultaneously into block diagonal form, and each block would have characteristic polynomial $f$. In his paper \cite{kip}, Kippenhahn proved that his conjecture holds for $n \leq 2$. Shapiro extended the validity range of the conjecture in a series of 1982 papers. In her first paper \cite{shapiro1} she demonstrated that if $f$ has a linear factor of multiplicity greater than $n/3$, then the conjecture holds. This proves the conjecture for $n = 3$. Her second paper \cite{shapiro2} shows that if $f = g^{n/2} $ where $g$ is quadratic, then the conjecture holds. This, combined with \cite{shapiro1}, proves the conjecture for $n = 4 $ and 5. Her final paper \cite{shapiro3} showed that the conjecture holds if $f$ is a power of a cubic factor. This is sufficient to prove the conjecture for order 6 in the form we are interested in, where $f$ is a power of an irreducible polynomial, but not Kippenhahn's orginal form. Buckley recently gave a proof of the same result as a corollary to a more general result about Weirstrass cubics \cite{buckley}. 

In his 1987 paper \cite{laffey}, Laffey disproved the simple form of the conjecture for $ n = 8$ with a single counterexample: \begin{equation}\label{laffey_ex}\begin{split} H &=  {\small \left( \begin{array}{cccccccc}
-122 & 0 & 12 & 18 & -30 & 18 & 26 & 10 \\
0 & -122 & -6 & -12 & -16 & -28 & 20 & -16 \\
12 & -6  & -218 & 0 & 44 & 8 & 24 & 12 \\
18 & -12 & 0 & -218 & -2 & -34 & -10 & 22 \\
-30 & -16 & 44 & -2 & -216 & 0 & -12 & -8 \\
18 & -28 & 8 & -34 & 0 & -216 & -8 & 36 \\
26 & 20 & 24 & -10 & -12 & -8 & -120 & 0 \\
10 & -16 & 12 & 22 & -8 & 36 & 0 & -120 
\end{array} \right)} ,  \\  &\qquad \qquad \; K ={\small \left( \begin{array}{cccccccc}
-4 & 0 & 0 & 0 & 0 & 0 & 0 & 0 \\
0 & -4 & 0 & 0 & 0 & 0 & 0 & 0 \\
0 & 0 & 4 & 0 & 0 & 0 & 0 & 0 \\
0 & 0 & 0 & 4 & 0 & 0 & 0 & 0 \\
0 & 0 & 0 & 0 & -8 & 0 & 0 & 0 \\
0 & 0 & 0 & 0 & 0 & -8 & 0 & 0 \\
0 & 0 & 0 & 0 & 0 & 0 & 8 & 0 \\
0 & 0 & 0 & 0 & 0 & 0 & 0 & 8
\end{array} \right)}. \end{split} \end{equation}

In 1998 Li, Spitkovsky, and Shukla disproved Kippenhahn's more general form of the conjecture for $n = 6$ by constructing a family of counterexamples of the form $f = \mbox{det}(I + xH + yK) = g^{2}h $, where both $g$ and $h$ are quadratics \cite{li}. 

We have used our results to construct an one-parameter family of counterexamples of Kippenhahn's conjecture, for square matrices of order $ n \geq 8$, in Section 4. This places Laffey's single counterexample into a wider context, and provides a novel way for constructing polynomial matrices which non-trivially have square determinant. We refer to \cite{kv} for a positive solution to a quantized version of Kippenhahn's conjecture.

Linear matrix polynomials, also called linear pencils, are 
a key tool in matrix theory and numerical analysis
(e.g. the generalized eigenvalue problem), 
and they frequently appear in (real) algebraic geometry (cf. \cite{vinnikov,NPT}).
Furthermore, linear pencils whose coefficients are hermitian matrices give rise
to linear matrix inequalities (LMIs). 
LMIs produce feasible regions of semidefinite programs (SDPs) \cite{sdp},
which are currently a hot topic in mathematical optimization.

\section{Obstacles to generating the full algebra}\label{obstacles}

In this section we prove Theorem \ref{not_gen_full_thm}, and another theorem clearly linking the Burnside graph of a set of matrices to Burnside's theorem for matrix algebras. To begin with, we add an extra definition pertaining to Burnside graphs, and present an example.

\begin{defn}[Strongly connected component]
A \textbf{strongly connected component} of a Burnside graph $B(A)$ is a maximal subset of nodes in which every pair of nodes are path connected in both directions.

\end{defn}

\begin{example}
Suppose that $$A_{1} = \left( \begin{array}{cccccc}
1 & 1 & 0 & 0 & 1 & 0 \\
0 & 1 & 0 & 1 & 0 & 0 \\
0 & 0 & 1 & 1 & 1 & 1 \\
1 & 1 & 0 & 1 & 1 & 0 \\
1 & 0 & 0 & 1 & 0 & 0 \\
1 & 0 & 1 & 0 & 1 & 0 
\end{array} \right) \mbox{ and } A_{2} = \left( \begin{array}{cccccc}
1 & 0 & 0 & 0 & 0 & 0 \\
0 & 1 & 0 & 0 & 0 & 0 \\
0 & 0 & 1 & 0 & 0 & 0 \\
0 & 0 & 0 & 2 & 0 & 0 \\
0 & 0 & 0 & 0 & 2 & 0 \\
0 & 0 & 0 & 0 & 0 & 2 
\end{array} \right).$$ \\ \\ The Burnside graph of $\left\lbrace A_{1}, A_{2} \right\rbrace $ is as follows: \begin{center}
\includegraphics[scale=.30]{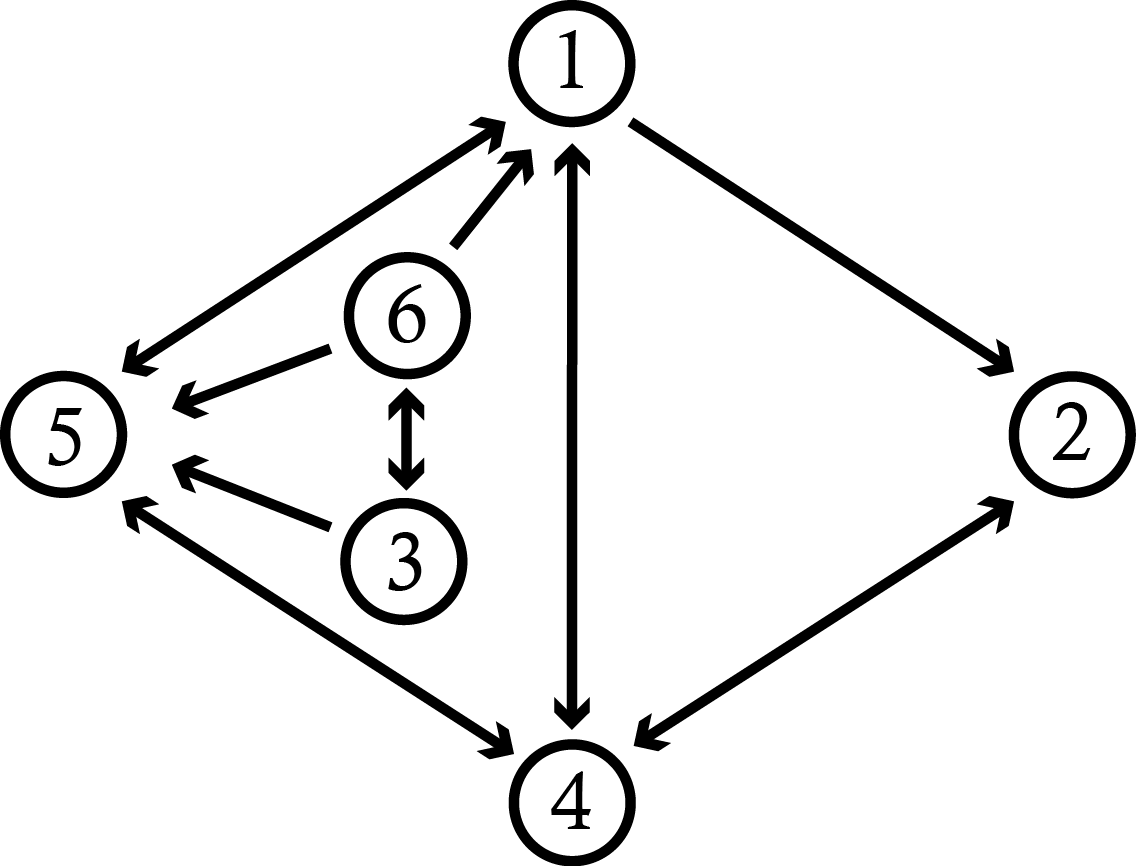}
\end{center}
The graph $B(A)$ is not strongly connected. Nodes 6 and 3 form a strongly connected component, but there is no way to get to these nodes from the rest of the graph. The strongly connected components are  $\left\lbrace 1, 2, 4, 5 \right\rbrace$ and $\left\lbrace 6, 3 \right\rbrace$.
\end{example}

Now we will prove Theorem \ref{not_gen_full_thm}: \begin{proof}[Proof of Theorem \ref{not_gen_full_thm}]
Since $B(A)$ is not strongly connected, there are two possibilities: \begin{enumerate}
\item $B(A)$ consists of at least 2 disconnected components,
\item $B(A)$ is connected but not strongly connected.
\end{enumerate} Strictly speaking, both of these cases can be dealt with at once, but we have separated them for clarity. \\ \emph{Case 1}: The graph can be split into two subsets which do not connect to each other. Renumbering the nodes corresponds to symmetrically permuting the columns and rows of the set of matrices $A$ and so does not affect the dimension of $\mathcal{A}$. Suppose then that nodes $1,...,j$ form one subset, and nodes $j+1,...,k$ form the other, and there are no edges joining one subset to the other. This means that, for every $A_{i}$: \begin{itemize}
\item for each row up to and including row $j$, every entry beyond the $j^{th}$ column is zero,
\item for each column up to and including column $j$, every entry beyond the $j^{th}$ row is zero.
\end{itemize}  This means that every matrix $A_{i}$ is in block diagonal form, and so $\mathcal{A} \neq M_{n}(\mathbb{C})$. By way of illustration, if $A_{1}$ and $A_{2}$ are as follows: $$A_{1} = \left( \begin{array}{cccccc}
1 & 0 & 0 & 0 & 1 & 0\\
0 & 0 & 1 & 0 & 0 & 0\\
0 & 0 & 0 & 0 & 0 & 0\\
0 & 0 & 0 & 1 & 0 & 0\\
0 & 0 & 0 & 0 & 0 & 0\\
1 & 0 & 0 & 0 & 0 & 1
\end{array} \right), A_{2} = \left( \begin{array}{cccccc}
0 & 0 & 0 & 0 & 0 & 0\\
0 & 0 & 0 & 1 & 0 & 0\\
0 & 0 & 0 & 1 & 0 & 0\\
0 & 1 & 0 & 0 & 0 & 0\\
1 & 0 & 0 & 0 & 0 & 0\\
0 & 0 & 0 & 0 & 1 & 0
\end{array} \right), $$ then the corresponding Burnside graph (and its permutation) is\begin{center}
\includegraphics[scale=.30]{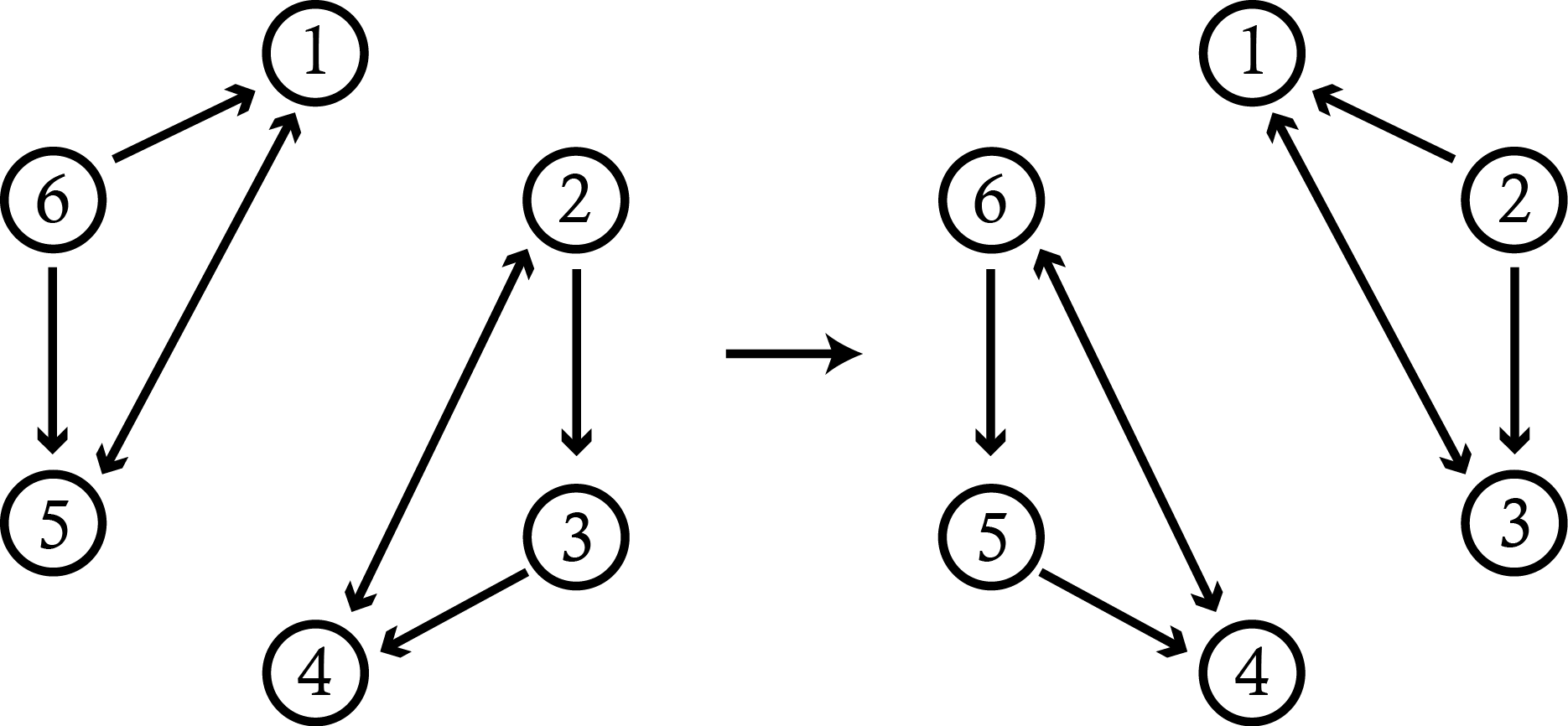}
\end{center} leading to the following block diagonal matrices obtained for the $A_{i}$: $$\tilde{A}_{1} = \left( \begin{array}{ccc|ccc}
1 & 0 & 1 & 0 & 0 & 0\\
1 & 1 & 0 & 0 & 0 & 0\\
0 & 0 & 0 & 0 & 0 & 0\\
\hline
0 & 0 & 0 & 1 & 0 & 0\\
0 & 0 & 0 & 0 & 0 & 0\\
0 & 0 & 0 & 0 & 1 & 0
\end{array} \right), \tilde{A}_{2} = \left( \begin{array}{ccc|ccc}
0 & 0 & 0 & 0 & 0 & 0\\
0 & 0 & 1 & 0 & 0 & 0\\
1 & 0 & 0 & 0 & 0 & 0\\
\hline
0 & 0 & 0 & 0 & 0 & 1\\
0 & 0 & 0 & 1 & 0 & 0\\
0 & 0 & 0 & 1 & 0 & 0
\end{array} \right). $$ \\ \\ \emph{Case 2}: Organise $B(A)$ into strongly connected components, which may consist of a single node. Reorder columns and rows so that these strongly connected components consist of consecutive nodes. Reordering columns and rows will not affect the dimension of the algebra $\mathcal{A}$.

If every strongly connected component had at least one inbound edge and at least one outbound edge, then a cycle would exist and $B(A)$ would be strongly connected. Therefore, without loss of generality we can assume that there is at least one node with no inbound edge. By symmetric reordering of rows and columns we can suppose that this is the first strongly connected component, associated with nodes $\left \lbrace 1,...,j \right \rbrace$, where $j$ may be equal to 1. Recall how edges are defined: a directed edge $i \rightarrow j$ exists if and only if there is at least one matrix in $A$ with a non-zero entry at the $(i,j)$ position. Therefore, since the first strongly connected component is not the endpoint for any edge, every matrix in $A$ must have all zeros in the first $j$ columns, beyond the $j^{th}$ position: $$ \left( \begin{array}{c|c}
j \times j  & j \times (n-j)  \\
\, & \, \\
\hline \\
\makebox(0,15){\text{\huge0}} & (n-j) \times (n-j)
\end{array} \right). $$ In other words, it must be possible to rearrange the matrix rows and columns into this block form, with a zero block in the lower left. Therefore, the subspace spanned by the first $j$ canonical basis vectors $\left \lbrace e_{1}, e_{2},...,e_{j} \right \rbrace$ of $F^k$ is invariant under every matrix in $A$ and therefore also invariant under the algebra $\mathcal{A}$, and so $\mathcal{A} \neq M_{n}(\mathbb{C})$. For example: $$ A_{1} = \left( \begin{array}{c|ccc}
0 & 0 & 0 & 0 \\
1 & 0 & 1 & 0 \\
1 & 0 & 0 & 1 \\
\hline 0 & 0 & 0 & 0
\end{array} \right), A_{2} =  \left( \begin{array}{c|ccc}
0 & 0 & 0 & 1 \\
0 & 1 & 0 & 0 \\
0 & 0 & 0 & 0 \\
\hline 0 & 0 & 0 & 0
\end{array} \right), $$ \begin{center}
\includegraphics[scale=.30]{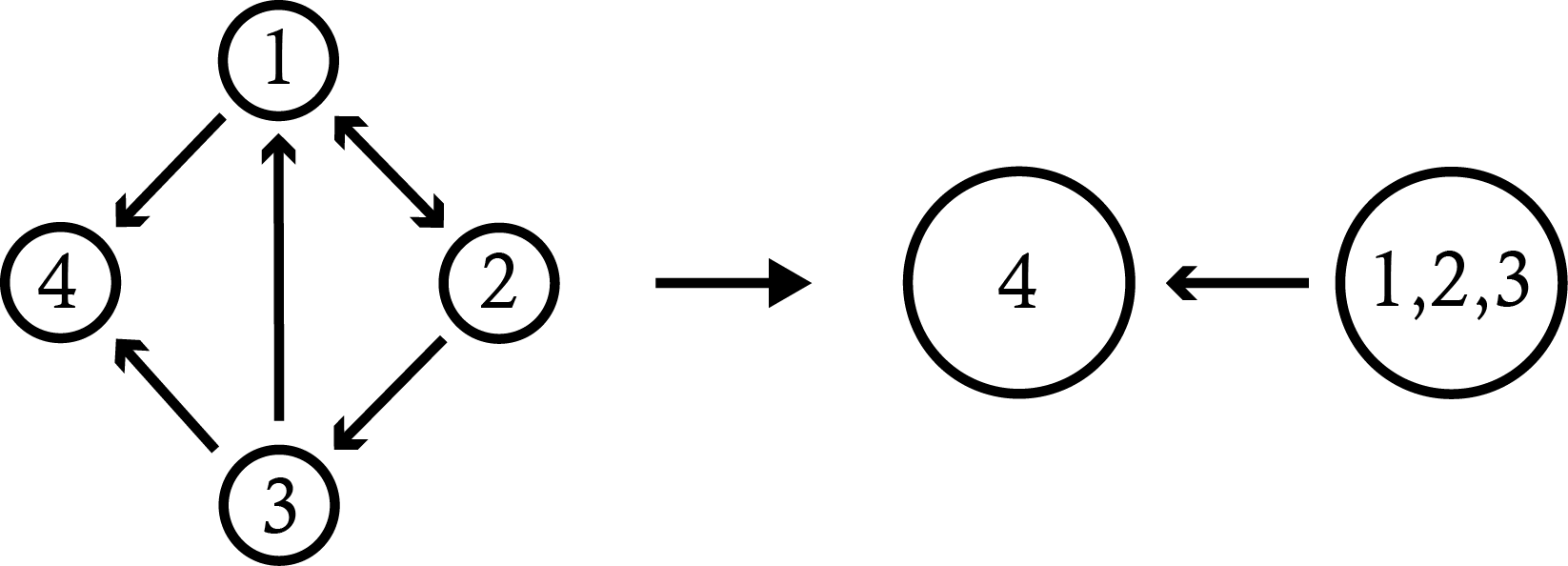}
\end{center} To handle the situation where the strongly connected component has no outbound edge, take the transpose of every $A_{i}$, which will not affect the dimension of $\mathcal{A}$, and repeat the argument.
\end{proof}

Now we will see the origin of the term Burnside graph. An algebra of linear transformations on a vector space is \emph{irreducible} if the only invariant subspaces with respect to the algebra are $\left\lbrace 0 \right \rbrace$ and the entire vector space. Recall Burnside's theorem for matrix algebras: 

\begin{theorem}[Burnside \cite{burnside_proof}]\label{burnside}
Let $V$ be a finite dimension vector space over $\mathbb{C}$, with $\dim (V) > 1$. The only irreducible algebra of linear transformations on $V$ is the full algebra $M_{\dim (V)}(\mathbb{C})$.
\end{theorem}

We have then the following corollary.

\begin{cor}\label{burnside_decomposition}
Every matrix algebra $\mathcal{A}$ over $\mathbb{C}$ can be put into a block upper triangular form, where the diagonal blocks are full sub-algebras $M_{n_{i}}(\mathbb{C})$ for some $n_{i} \in \mathbb{N}$.
\end{cor}

\begin{proof}
Let $V$ be the vector space upon which $\mathcal{A}$ acts. Take the smallest $\mathcal{A}$-invariant subspace of $V$, denoted $U_{1}$. Then we have $V = U_{1} \oplus U^{\perp}$.  Note that if the smallest invariant subspace is all of $V$, then straight away we have that $\mathcal{A}$ is the full algebra (by Theorem \ref{burnside}), which would trivially satisfy the corollary. So assume that $U_{1}$ is a proper subspace of $V$.

Let $u_{1}$ be an orthonormal basis for $U_{1}$, and $v_{1}$ and orthonormal basis for $ U^{\perp}$. Then $( u_{1}, v_{1} )$ is an orthonormal basis for $V$. With respect to this basis, then, $\mathcal{A}$ must have the form $$\left( \begin{array}{ccc|ccccc}
\, & \uparrow & \,  & \; & \; & \; & \; & \; \\
\longleftarrow & \dim (U_{1}) & \longrightarrow  & \; & \; & \hdots & \; & \; \\
\, & \downarrow & \,  & \; & \; & \; & \; & \; \\ \hline
\, & \, & \,  & \; & \; & \; & \; & \;   \\ 
\, & \makebox(0,0){\text{\huge0}} & \,  & \; & \; & \hdots & \; & \; \\
\, & \, & \,  & \; & \; & \; & \; & \; \\
\end{array} \right) $$ to ensure that $U_{1}$ is invariant. Examine the $\dim (U_{1})$ block in the top left. Since we have assumed that $U_{1}$ is the smallest invariant subspace of $V$, there can be no $\mathcal{A}$ invariant subspaces of $U_{1}$. Treating this upper left block as a sub-algebra acting on $u_{1}$ embedded within $V$, Burnside's theorem then tells us that this sub-algebra can only be the full algebra $M_{\dim (U_{1})}(\mathbb{C})$: $$ \mathcal{A} = \left( \begin{array}{ccc|ccccc}
\, & \, & \,  & \; & \; & \; & \; & \; \\
\, & M_{\dim (U_{1})}(\mathbb{C}) & \,  & \; & \; & \hdots & \; & \; \\
\, & \, & \,  & \; & \; & \; & \; & \; \\ \hline
\, & \, & \,  & \; & \; & \; & \; & \;   \\ 
\, & \makebox(0,0){\text{\huge0}} & \,  & \; & \; & \hdots & \; & \; \\
\, & \, & \,  & \; & \; & \; & \; & \; \\
\end{array} \right). $$ Repeating this process for $U_{1}^{\perp}$, we can work our way through all of $V$ and put $\mathcal{A}$ in the required block upper triangular form, with copies of the full algebra of various sizes down the diagonal: $$\mathcal{A} = \left( \begin{array}{ccccc}
M_{\dim (U_{1})}(\mathbb{C}) & \hdots & \hdots & \hdots & \hdots  \\
\, &  M_{\dim (U_{2})}(\mathbb{C}) & \hdots & \hdots & \hdots \\
\, & \, & M_{\dim (U_{3})}(\mathbb{C}) & \hdots & \hdots \\
\, & \makebox(0,0){\text{\huge0}} & \, & \ddots & \, \\
\, & \, & \, & \, & M_{\dim (U_{n})}(\mathbb{C})

\end{array} \right). $$
\end{proof}

The connection with the Burnside graph is now clear. Putting a set of matrices $\mathcal{A}$ into what we may call Burnside form, as in Corollary \ref{burnside_decomposition}, gives a Burnside graph where the strongly connected components correspond to the diagonal blocks, and the non-zero upper blocks correspond to the connections between strongly connected components, making allowance for left and right invariance corresponding to inbound and outbound edges.

\begin{cor}\label{sym_H_cor}
If every $A_{i}$ in $A$ is hermitian, and $B(A)$ is not strongly connected, then every $A_{i}$ in $A$ can be put simultaneously in block diagonal form with an orthogonal transformation.
\end{cor}

\begin{proof}
If every matrix in $A$ is hermitian, then every element $A_{i,jk}$  of each matrix $A_{i}$ contributes the same edge to $B(A)$ as does $A_{i,kj}$, but with the direction reversed. Therefore in this case $B(A)$ will partition into strongly connected components, because there will be no one-way edges. This means that, for hermitian $A$, if $B(A)$ as a whole is not strongly connected, it must be disconnected, and as Case 1 in the proof of Theorem \ref{not_gen_full_thm}, a basis exists which simultaneously block-diagonalises every $A_{i}$. We can make the disconnected nature of $B(A)$ visible with an orthogonal transformation. 
\end{proof}

\section{Burnside graphs and their associated algebras}\label{algebras}

Laffey \cite{laffey_struct} shows for a pair of matrices, one of which is  diagonal with distinct eigenvalues, a strongly connected Burnside graph is all that is needed to guarantee the generation of the full matrix algebra. Without this distinct eigenvalue assumption, additional conditions are necessary. In this section, we provide a set of conditions on the submatrix blocks which guarantee generation of the full algebra. Of particular interest is the case where a pair of matrices have eigenvalues all of multiplicity 2. This is related to Kippenhahn's conjecture \cite{kip}, to which we will construct a family of counterexamples in Section 4. 
 
First we will need some definitions, which will allow us to prove an expanded version of Theorem \ref{2_gens_thm_simple} given in the introduction, which ensures that the algebra $\mathcal{A}$ generated by a set of matrices $A$ is the full matrix algebra. \begin{defn}[$p$-block]\label{p-block}
Let $p \in \mathbb{N}$, and let $H$ be a real symmetric matrix of size $pn \times pn$. A \textbf{p-block} of $H$ is a $p \times p$ submatrix occupying columns $p(i-1) + 1 $ to $pi $ and rows $p(j-1) + 1 $ to $pj $ for some $i,j = 1,...,n$. We denote such a $p$-block by $H_{ij}$ and say that it is in the $ij$-position in $H$.
\end{defn} 

\begin{defn}[$p$-word]\label{p-word}
Given a set $\lbrace H^{(1)},...,H^{(m)} \rbrace $ of $pn \times pn$ matrices over a field $F$, a \textbf{p-word} based at i$_{1}$ and ending at i$_{q}$ is a matrix product $$H^{(k_{1})}_{i_{1}j_{1}}H^{(k_{2})}_{j_{1}i_{1}}....H^{(k_{q})}_{j_{q-1}i_{q}},$$ where the second index of each entry matches the first index of the subsequent entry. For example $$ H^{(1)}_{12}H^{(1)}_{24}H^{(2)}_{43}H^{(3)}_{33} $$ is a $p$-word based at 1 and ending at 3 over the matrices $\lbrace H^{(1)}, H^{(2)},H^{(3)} \rbrace$.
\end{defn}

\begin{defn}[Condition Mult$_{p}$]\label{mult-q}
Let $p \in \mathbb{N}$. A hermitian matrix $K$ satisfies \textbf{Condition Mult$_{p}$} if its eigenvalues have maximum multiplicity $p$. For example, if $K$ has eigenvalues $\left \lbrace 1, -1, 2, 3, 4, 4 \right \rbrace $, then $K$ satisfies Condition Mult$_{2}$.
\end{defn} 

\begin{defn}[Condition $L-p$]\label{L-p}
Let $p \in \mathbb{N}$, and let $H$ be a hermitian matrix of size $pn \times pn$. Take a partition of $n$ as $(l_{1} = 1,l_{2} = 1, l_{3}, l_{4}....,l_{m})$ so that $\sum_{i = 1}^{m} l_{i} = n $. 

Suppose also that each $l_{j} = \sum_{i = 1}^{k} l_{i}$ for some $k$, with $k < j$. For example, $(1, 1, 2, 4, 8)$ or $(1, 1, 1, 1)$ or $(1, 1, 2, 2)$.

Identify \textbf{square} blocks along the top row of $H$, of non-decreasing size $pl_{1} = p, pl_{2} = p, pl_{3}, pl_{4}...,pl_{m}$, starting at the top left and proceeding along to the right. Note that the importance of $p$ is that it sets the minimum size of the smallest pair of blocks with which the block sequence starts.

If such a partition exists so that each of these blocks is invertible, we say that $H$ satisfies \textbf{condition $L-p$}. 

\end{defn} An example to illustrate these partitions:

\begin{example}
Consider the matrix $$ H = \left( \begin{array}{cc|cc|cccc|cccc}
0 & 1 & 1 & 1     & 1 & 0 & 1 & 0 & 1 & 1 & 1 & 0 \\
1 & 0 & 0 & 1     & 0 & 1 & 1 & 1 & 0 & 1 & 0 & 1 \\
\, & \, & \, & \, & 0 & 1 & 1 & 0 & 0 & 0 & 1 & 1 \\
\, & \, & \, & \, & 1 & 0 & 0 & 1 & 1 & 1 & 0 & 1 \\
\, & \vdots & \, & \vdots & \, & \vdots & \, & \, & \, & \vdots & \, & \, \\
\end{array} \right). $$ Only some entries of $H$ are shown for clarity. Let $p = 2$, and take the partition $(1, 1, 2, 2)$ of 6. Then $H$ satisfies condition $L-2$. We could also have taken the partition $(1, 1, 1, 1, 1, 1)$ and $H$ would still have satisfied Condition $L-2$.
\end{example} In this example, notice how we could have used two different partitions of 6. The point is that the condition requires only the existence of a suitable partition. It does not refer to a specific partition.

We will now we can present the full version of Theorem \ref{2_gens_thm_simple}.

\begin{theorem}[Even-order constructibility]\label{2_gens_thm}
Let $H$ and $K$ be $2n \times 2n$ hermitian matrices over $\mathbb{C}$. Then if, in the basis in which $K$ is diagonal with weakly ascending diagonal entries, \begin{enumerate}

\item $K$ satisfies Condition Mult$_{2}$,

\item $H$ satisfies Condition $L-2$,
\item there exist distinct 2-words $w_{1} $ and $w_{2} $ both based at 1 so that $w_{1}w_{1}^{T}$ and $w_{2}w_{2}^{T}$ do not commute,
\end{enumerate} then the algebra $\mathcal{A}$ generated by $H$ and $K$ is the full algebra $M_{2n}(\mathbb{C})$. \\
\end{theorem} 

\begin{proof}

Since $K$ satisfies Condition Mult$_{2}$, we can find an orthogonal transformation to put $K$ in the form $$K = \left( \begin{array}{cccc}
k_{1} & \, & \, & \, \\
\, & k_{2} & \, & \, \\
\, & \, & \ddots & \, \\
\, & \, & \, & k_{2n}
\end{array} \right), $$ where each of the $k_{i}$ appear with multiplicity at most 2, and there is at least one such pair. With a symmetric permuation of rows and columns, we can arrange for all such pairs eigenvalues to be consecutive.

Since the diagonal entries $\lbrace k_{1},..., k_{2n} \rbrace$ of $K$ now all lie along the diagonal and all have multiplicity at most 2, with there being at least one such pair which must be consecutive (due to the ascending order assumption), we can define $n$ distinct polynomials $\lbrace q_{1},..., q_{n} \rbrace$ of order $2n-1$ such that  $$q_{i}(K) = \left( \begin{array}{cccc}
0_{2} & 0_{2} & \hdots & 0_{2} \\
\vdots & \ddots & \hdots & \vdots \\
\vdots & \, & I_{2} & \vdots \\
0_{2} & \hdots & \hdots & \ddots
\end{array} \right) \in \mathcal{A}, $$ that is, for every $i = 1,...,n$, $q_{i}(K)$ has $I_{2}$ at the $(i,i)$ 2-block position and zeroes elsewhere and such matrices are elements of the algebra $\mathcal{A}$.

Now consider $q_{i}(K) \, H \, q_{j}(K)$. Conjugation in this way produces a matrix of the form $$\left( \begin{array}{cccc}
0_{2} & 0_{2} & \hdots & 0_{2} \\
\vdots & \ddots & H_{ij} & \vdots \\
\vdots & \, & \ddots & \vdots \\
0_{2} & \hdots & \hdots & \ddots
\end{array} \right) \in \mathcal{A}, $$ a matrix which consists of only the $(i,j)$ 2-block of $H$ at the $(i,j)$ location, with zeros elsewhere. Since we are conjugating by elements of $\mathcal{A}$, these isolated 2-blocks are themselves elements of $\mathcal{A}$. The 2-words $w_{1}$ and $w_{2}$ can therefore be obtained by isolating each 2-block in each word and multiplying the word out. Since $H$ is symmetric, $w^{T}_{1}$ and $w^{T}_{2}$ are also both valid 2-words. Both $w_{1}w^{T}_{1}$ and $w_{2}w^{T}_{2}$ are symmetric, and moreover lie in the $(1,1)$ position. We have assumed that $w_{1}w^{T}_{1}$ and $w_{2}w^{T}_{2}$ do not commute, and so by considering available dimensions we can see that $w_{1}w^{T}_{1}$ and $w_{2}w^{T}_{2}$ generate $M_{2}(\mathbb{C})$. Therefore, $$ \left( \begin{array}{ccc}
M_{2}(F) & \hdots & \hdots \\
\vdots & \, & \,  \\
\vdots & \, & \, 
\end{array} \right) \subseteq \mathcal{A}. $$ Now we can use Condition $L-2$ of $H$.

Condition $L-2$ requires that $H_{12}$ is invertible and  therefore so is $H_{21}$ (since $H$ is hermitian), and so we can move the copy of $M_{2}(F)$ around by multiplying with $$\left( \begin{array}{cccc}
0_{2} & H_{12} & \, & \, \\
0_{2} & 0_{2} & \, & \, \\
\vdots & \, & \ddots & \, \\
\, & \, & \, & \ddots
\end{array} \right), \left( \begin{array}{cccc}
0_{2} & 0_{2} & \, & \, \\
H_{21} & 0_{2} & \, & \, \\
\vdots & \, & \ddots & \, \\
\, & \, & \, & \ddots
\end{array} \right) \in \mathcal{A} $$ as follows: $$ \left( \begin{array}{ccc}
M_{2}(\mathbb{C}) & \hdots & \hdots \\
\vdots & \, & \,  \\
\vdots & \, & \, 
\end{array} \right) \left( \begin{array}{cccc}
0_{2} & H_{12} & \, & \, \\
0_{2} & 0_{2} & \, & \, \\
\vdots & \, & \ddots & \, \\
\, & \, & \, & \ddots
\end{array} \right) = \left( \begin{array}{ccc}
0_{2} & M_{2}(\mathbb{C}) & \hdots \\
\vdots & \, & \,  \\
\vdots & \, & \, 
\end{array} \right), $$ $$ \left( \begin{array}{cccc}
0_{2} & 0_{2} & \, & \, \\
H_{21} & 0_{2} & \, & \, \\
\vdots & \, & \ddots & \, \\
\, & \, & \, & \ddots
\end{array} \right) \left( \begin{array}{ccc}
M_{2}(\mathbb{C}) & \hdots & \hdots \\
\vdots & \, & \,  \\
\vdots & \, & \, 
\end{array} \right)  = \left( \begin{array}{ccc}
0_{2} & \hdots & \hdots \\
M_{2}(\mathbb{C}) & \, & \,  \\
\vdots & \, & \, 
\end{array} \right), $$ \begin{align*} \left( \begin{array}{cccc}
0_{2} & 0_{2} & \, & \, \\
H_{21} & 0_{2} & \, & \, \\
\vdots & \, & \ddots & \, \\
\, & \, & \, & \ddots
\end{array} \right) &\left( \begin{array}{ccc}
M_{2}(\mathbb{C}) & \hdots & \hdots \\
\vdots & \, & \,  \\
\vdots & \, & \, 
\end{array} \right) \left( \begin{array}{cccc}
0_{2} & H_{12} & \, & \, \\
0_{2} & 0_{2} & \, & \, \\
\vdots & \, & \ddots & \, \\
\, & \, & \, & \ddots
\end{array} \right) \\ &= \left( \begin{array}{ccc}
0_{2} & 0_{2} & \hdots \\
0_{2} & M_{2}(\mathbb{C}) & \,  \\
\vdots & \, & \, 
\end{array} \right). \end{align*} Since $H_{12}$ and its transpose $H_{21}$ is invertible, each of these products is equal to $M_{2}(\mathbb{C})$. Thus there are copies of $M_{2}(\mathbb{C})$ in the $11$, $12$, $21$, and $22$ positions. Therefore every possible $4 \times 4$ matrix over $F$ exists in the top-left corner, since every such matrix can be partitioned into 4 blocks which are elements of $M_{2}(\mathbb{C})$. Therefore $$ \left( \begin{array}{ccc}
M_{4}(\mathbb{C}) & \hdots & \hdots \\
\vdots & \, & \,  \\
\vdots & \, & \, 
\end{array} \right) \subseteq \mathcal{A}. $$ Condition $L-2$ ensures that we can simply keep repeating the process. If the next invertible block is of size 1, we repeat the process on the copy $M_{2}(\mathbb{C})$ which lies in the top left. If on the other hand it is of size 4, we use our new copy of $M_{4}(\mathbb{C})$. The entire matrix is filled out in this way with copies of each full algebra of lesser order.

Therefore $\mathcal{A} = M_{2n}(\mathbb{C})$.
\end{proof}

\begin{example}
Let $C_{1} = \left( \begin{array}{cccc}
0 & 1 & 0 & 1 \\
1 & 0 & 2 & 0 \\
0 & 2 & 0 & 0 \\
1 & 0 & 0 & 0
\end{array} \right) $ and $C_{2} = \left( \begin{array}{cccc}
1 & 0 & 0 & 0 \\
0 & 1 & 0 & 0 \\
0 & 0 & 0 & 0 \\
0 & 0 & 0 & 0
\end{array} \right).$ \\ \\ Let $\mathcal{C}$ be the algebra generated by $C_{1}$ and $C_{2}$. The Burnside graph $B(\lbrace C_{1}, C_{2} \rbrace )$ is as follows: \begin{center}
\includegraphics[scale=.20]{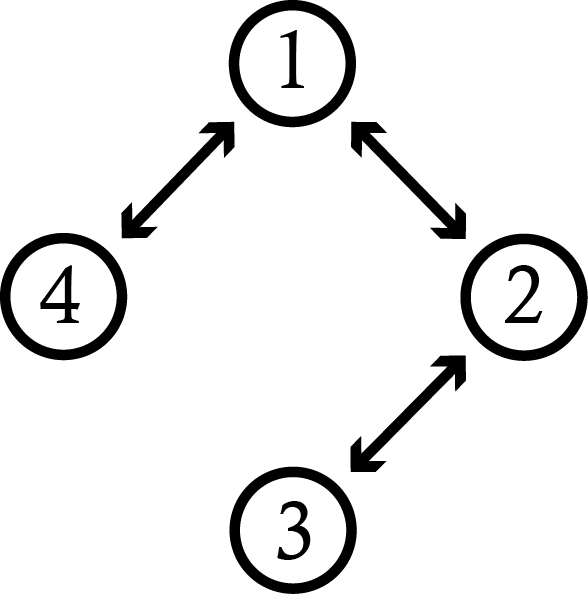}
\end{center}
We can check off the requirements of Theorem \ref{2_gens_thm} one by one: \begin{enumerate}
\item $C_{2}$ satsifies condition Mult$_{2}$,
\item $\left( \begin{array}{cc}
0 & 1 \\
1 & 0
\end{array} \right) $ and $\left( \begin{array}{cc}
0 & 1 \\
2 & 0
\end{array} \right)$ are both invertible, and so $C_{1}$ satisfies condition L-2,
\item $\left( \begin{array}{cc}
0 & 1 \\
1 & 0
\end{array} \right) $ and $\left( \begin{array}{cc}
0 & 1 \\
2 & 0
\end{array} \right) \left( \begin{array}{cc}
0 & 2 \\
1 & 0
\end{array} \right) =  \left( \begin{array}{cc}
1 & 0 \\
0 & 4
\end{array} \right)$ do not commute.

\end{enumerate} Therefore by Theorem \ref{2_gens_thm}, $\mathcal{C} = M_{4}(\mathbb{C})$.
\end{example}

The key element of Theorem \ref{2_gens_thm} is the non-commutativity requirement. This is what sets the process going by giving us a sub-algebra to start with. We can obtain a few more corollaries using the following theorem of Laffey:

\begin{theorem}[Laffey's Generation Theorem \cite{laffey_struct}]\label{laffey_gen_thm}
Let $A = \left\lbrace A_{1},...,A_{k}\right\rbrace$ be a set of $n \times n$ matrices over a field $F$. Let $B$ be an $n \times n$ diagonal matrix over a field $F$ with distinct diagonal entries. Let $\mathcal{\tilde{A}}$ be the algebra generated by $A \cup \lbrace B \rbrace$. Then the following statements are equivalent for $n > 1$: \begin{enumerate}
\item $\mathcal{\tilde{A}}$ is simple,
\item $\mathcal{\tilde{A}} = M_{n}(F)$,
\item $B(A)$ is strongly connected.
\end{enumerate}
\end{theorem}

Laffey's theorem  helps us establish the following theorem:

\begin{theorem}[q-block constructibility]\label{q_gens_thm_ver_2}

Let $H$ and $K$ be $qn \times qn$ hermitian matrices over $\mathbb{C}$. Then if, in the basis in which $K$ is diagonal with weakly ascending diagonal entries, \begin{enumerate}
\item There is a set $\left\lbrace v_{1},...,v_{k} \right\rbrace$ of $q$-words starting and ending at 1 such that the Burnside graph $B( \left\lbrace v_{1},..., v_{k} \right\rbrace )$ is strongly connected,
\item $K$ satisfies Condition Mult$_{q}$,
\item $H$ satisfies Condition $L-q$,
\item there is some $q$-word $w$ based at 1 so that $w w^{T}$ is diagonal and has $q$ unique eigenvalues,
\end{enumerate} then the algebra $\mathcal{A}$ generated by $H$ and $K$ is the full algebra $M_{2n}(\mathbb{C})$. 
\end{theorem}

\begin{proof}
Since $H$ is hermitian, $w w^{T}$  will also be hermitian and can therefore be diagonalised. Diagonalising $w w^{T}$ can be done by conjugating $H$ and $K$ with the Kronecker product $P \otimes I_{n}$, where $P$ is the unitary diagonalisation matrix of $w w^{T}$. Because $K$ itself is a Kronecker product $I_{q} \otimes D$, where $D$ is some diagonal matrix, it will not be affected. We have assumed that $w w^{T}$ has unique eigenvalues - denote them by $w_{1},...,w_{q}$. Use the same block-isolation process as in the proof of Theorem \ref{2_gens_thm}, to obtain that $$ \left( \begin{array}{ccc|ccc}
w_{1} & \, & \, & \, & \, & \, \\
\, & \ddots & \, & \, & \, & \,\\
\,  & \, & w_{q} & \, & \, & \,\\
\hline  & \, & \, & \, & \, & \,\\
\,  & \, & \, & \, & \, & \, \\
\end{array} \right) \in \mathcal{A}, \left( \begin{array}{c|ccc}
v_{j} & \, & \, & \, \\
\hline  & \, & \, & \, \\
\,  & \, & \, & \, \\
\end{array} \right) \in \mathcal{A},$$ for each $j=1,...,k$. Since we have assumed that $B(\left\lbrace v_{1},...,v_{k} \right\rbrace )$ is strongly connected, Laffey's Generation Theorem gives us a copy of $M_{q}(\mathbb{C})$ in the top left corner: $$ \left( \begin{array}{ccc}
M_{q}(\mathbb{C}) & \hdots & \hdots \\
\vdots & \, & \,  \\
\vdots & \, & \, 
\end{array} \right) \subseteq \mathcal{A}. $$ The rest of the proof follows as in Theorem \ref{2_gens_thm}, using Condition $L-q$ to distribute copies of $M_{q}(\mathbb{C})$ all around the remaining positions.

\end{proof}

Theorem \ref{not_gen_full_thm} actually gives a condition that is in a generic sense necessary and sufficient. Recall the notion of `generic' from algebraic geometry: a property holds generically if it holds except on a proper Zariski-closed subset.

\begin{cor} 
Let $A = \{A_1 , \ldots, A_k\}$, where $k > 1$, be a set
of $n \times n$ matrices over $\mathbb{C}$ with strongly connected Burnside graph $B(A)$. Let $\mathcal A$ denote the algebra generated by $A$. If $A$ is a generic tuple, then $\mathcal{A} = M_n(\mathbb{C})$.
\end{cor}

\begin{proof}
If $\mathcal A \subsetneq M_n(\mathbb{C})$, then the $\sigma \times n^2$ matrix, where $\sigma = \frac{k^{n^{2} - 1} - k}{k - 1} $, obtained by forming all products of the $A_j$ matrices of length less than or equal to $n^2$ and flattening them has rank strictly less than $n^2$. This can be expressed with the vanishing of all $n^2\times n^2$ minors, so is a Zariski closed condition. Thus it suffices to find a tuple $A$ with the given graph $B(A)$ that generates the full matrix algebra.

By Laffey's Theorem \ref{laffey_gen_thm}, we can simply take an $n \times n$ diagonal matrix $A_{1}$ with $n$ distinct eigenvalues, and a single matrix $A_{2}$ which is the adjacency matrix of the strongly connected Burnside graph $B(A)$, and $\mathcal{A}(\left\lbrace A_{1},A_{2} \right\rbrace) =  M_n(\mathbb{C})$. 
\end{proof}

Likewise, condition (4) in Theorem \ref{q_gens_thm_ver_2} holds generically. We thus have:

\begin{cor}
Let $H$ and $K$ be real $qn \times qn$ symmetric matrices over $\mathbb{C}$ generic with respect to the following properties:
\begin{enumerate}
	\item There is a set $\left\lbrace v_{1},...,v_{k} \right\rbrace$ of $q$-words starting and ending at 1 such that the Burnside graph $B( \left\lbrace v_{1},..., v_{k} \right\rbrace )$ is strongly connected,
	\item $K$ satisfies Condition Mult$_{q}$,
	\item $H$ satisfies Condition $L-q$.
\end{enumerate}
Then the algebra generated by $H$ and $K$ is the full algebra $M_{qn}(\mathbb{C})$.
\end{cor}

\begin{proof}
The matrix $H_{1i} H_{i1}$ is $q\times q$ and the condition of having $q$ unique eigenvalues is Zariski open. It thus suffices to find an example where (1)-(3) hold and $H_{1i} H_{i1}$ has $q$ unique eigenvalues. We will construct a suitable $\tilde{H}$ and $\tilde{K}$. Take the Burnside graph $B( \left\lbrace v_{1},..., v_{k} \right\rbrace )$, and construct $\tilde{H}_{11}$ as the adjacency matrix of this graph. Place this $q \times q$ matrix in top left position of $\tilde{H}$. Then construct $\tilde{H}_{12} = \tilde{H}_{21} = \mbox{diag}(h_{1},...,h_{q})$, where the $h_{i}$ are distinct and positive. Place a copy of $\tilde{H}_{12}$ in every remaining position, to fill out $\tilde{H}$. Choose $n$ different eigenvalues $k_{i}$, and define $\tilde{K} = \mbox{diag}(k_{1}I_{q},...,k_{n}I_{q})$. Together $\tilde{H}$ and $\tilde{K}$ satisfy all the conditions of Theorem \ref{q_gens_thm_ver_2}, and so the theorem holds generically.
\end{proof}

\section{Kippenhahn's Conjecture}\label{kip_section}

Here we put Theorem \ref{2_gens_thm} to use, and construct a one-parameter family of counterexamples to Kippenhahn's conjecture.

\begin{remark}
The interest in Kippenhahn's conjecture can also be illustrated geometrically. Given a linear pencil $L=I+xH+yK$ as in Conjecture \ref{kip}, its determinant $f=\det L$ gives rise to the affine scheme $Spec \, \mathbb{C}[x,y]/(f)$. If condition (1) in Conjecture \ref{kip} holds, then then this scheme is obviously nonreduced - see for example Chapter 5, Section 3.4 of \cite{shaf}.
\end{remark}

\subsection{Existing counterexample}\label{existing}
 
Recall Laffey's counterexample \eqref{laffey_ex} from the introduction. It  satisfies the requirements of Theorem \ref{2_gens_thm}. In particular, $H$ satisfies Condition $L-2$, via the partition $(1, 1, 2)$ of 4. In the rest of this paper, we will use Theorem \ref{2_gens_thm} to present the construction of an entire family of counterexamples to the strong form of Kippenhahn's conjecture, of order 8 and above.

\subsection{Family of counterexamples for $8 \times 8 $ and above}\label{family}

Let us define: $$\alpha = \left( \begin{array}{cc}
1 & 0 \\
0  & -1
\end{array}  \right), \quad \beta = \left( \begin{array}{cc}
0 & b \\
b & 0
\end{array}  \right), \quad U = \left( \begin{array}{cc}
0 & -1 \\
1 & 0 
\end{array} \right), $$ where $b \in \mathbb{R}$ and $b \neq 0$. The key properties of these matrices are listed in the following lemma:

\begin{lemma}\label{props2}
These properties follow directly from the definitions of $\alpha$, $\beta$, and $U$:
\begin{enumerate}
\item $[\alpha, \beta] = \alpha\beta - \beta\alpha = \left( \begin{array}{cc}
0 & 2b \\
-2b & 0
\end{array} \right) \neq 0,$
\item $\alpha^{2} = -U^{2} = I_{2},$
\item $(\alpha + \beta)^{2} = (1 + b^{2})I_{2}, \mbox{ and },$
\item $\alpha U + U \alpha = \beta U + U \beta = 0.$
\end{enumerate}

\end{lemma}

Define $2n \times 2n$ matrices $$A = \left( \begin{array}{ccccccc}
U & \alpha & \alpha & \alpha + \beta & \alpha & \hdots & \alpha\\
-\alpha & 2U & \alpha & \alpha \\
-\alpha & -\alpha & 3U & 0  \\
-\alpha - \beta & -\alpha & 0 & 4U   \\
-\alpha & \, & \, & \, & 5U \\
\vdots & \, & \, & \,  & \, & \ddots \\
-\alpha & \, & \, & \,  & \, & \, & nU\\
\end{array}\right),\qquad B = \left( \begin{array}{cccc}
U  \\
\, & U  \\
\, & \, & \ddots  \\
\, & \, & \, & U
\end{array} \right). $$ Both $A$ and $B$ are skew-symmetric. Recall our notation; we denote $2 \times 2$ blocks of a matrix by an $ij$ subscript, so for example $A_{14} = \alpha + \beta$.

Now, define symmetric matrices $$H = A^{2}, \quad K = AB+BA.  $$ Use $\lbrace \; \; , \; \rbrace $ to denote the anti-commutator, and consider the block form of $K$: $$\left( \begin{array}{ccccccc}
2U^{2} &  \lbrace \alpha , U \rbrace & \lbrace \alpha , U \rbrace & \lbrace \alpha + \beta , U \rbrace & \lbrace \alpha , U \rbrace & \hdots & \lbrace \alpha , U \rbrace\\
-\lbrace \alpha , U \rbrace & 4U^{2} & \lbrace \alpha , U \rbrace & \lbrace \alpha , U \rbrace \\
-\lbrace \alpha , U \rbrace & -\lbrace \alpha , U \rbrace & 6U^{2} & \lbrace \alpha , U \rbrace \\
-\lbrace \alpha + \beta , U \rbrace & -\lbrace \alpha , U \rbrace & -\lbrace \alpha , U \rbrace & 8U^{2} \\
-\lbrace \alpha , U \rbrace & \, & \, & \, & 10U^{2} \\
\vdots & \, & \, & \,  & \, & \ddots \\
-\lbrace \alpha , U \rbrace & \, & \, & \,  & \, & \, & 2nU^{2}\\
\end{array}\right) $$ By Lemma \ref{props2}, $U\alpha = - \alpha U$, $U\beta = - \beta U$, and $U^{2} = -I_{2}$.   Therefore every off-diagonal block of $K$ vanishes, and $$K = \left( \begin{array}{cccc}
-2I_{2}  \\
\, & -4I_{2}  \\
\, & \, & \ddots  \\
\, & \, & \, & -2nI_{2}
\end{array} \right),$$ and likewise $H$ is of the form $${\small \left( \arraycolsep=2pt
\begin{array}{ccccccccccc}
 -b^2-n & 0 & -2 & b-1 & 1 & -2 & 3 b+1 & -3 & \hdots & 0 & 1-n \\
 0 & -b^2-n & -b-1 & -2 & -2 & 1 & -3 & 1-3 b & \, & 1-n & 0 \\
 -2 & -b-1 & -7 & 0 & -1 & -1 & -1 & -b-2 & \, & -1 & 0 \\
 b-1 & -2 & 0 & -7 & -1 & -1 & b-2 & -1 & \, & 0 & -1 \\
 1 & -2 & -1 & -1 & -11 & 0 & -2 & -b & \, & -1 & 0 \\
 -2 & 1 & -1 & -1 & 0 & -11 & b & -2 & \, & 0 & -1 \\
 3 b+1 & -3 & -1 & b-2 & -2 & b & -b^2-18 & 0 & \, & -1 & b \\
 -3 & 1-3 b & -b-2 & -1 & -b & -2 & 0 & -b^2-18 & \hdots & -b & -1 \\
 \vdots & \, & \, & \, & \, & \, & \, & \vdots & \ddots & \, & \vdots \\
 0 & 1-n & -1 & 0 & -1 & 0 & -1 & -b & \, & -1 -n^{2} & 0 \\
 1-n & 0 & 0 & -1 & 0 & -1 & b & -1 & \hdots & 0 & -1 -n^{2} \\
\end{array}
\right) }. $$ Note the regular structure of the 2-blocks of $H$. Where $j \geq 5$, the $H_{1j}$ and $H_{j1}$ 2-blocks  are of the form $\left( \begin{array}{cc}
0 & 1-j \\
1-j & 0
\end{array} \right) $, $H_{jj}$ blocks are of the form \\ $\left( \begin{array}{cc}
-1-j^{2} & 0 \\
0 & -1-j^{2}
\end{array} \right), $ and $H_{4j}$ and $H_{j4}$ blocks are of the form $\left( \begin{array}{cc}
-1 & b \\
b & -1
\end{array} \right)$. All other $H_{ij}$ 2-blocks where $i,j \geq 5$ are of the form $\left( \begin{array}{cc}
-1 & 0 \\
0 & -1
\end{array} \right)$. We will now show that $H$ and $K$ violate Kippenhahn's conjecture. \begin{lemma}\label{eigenpairs}
All of the eigenvalues of $xH + yK$ have even multiplicity.
\end{lemma}

\begin{proof}
Consider first $Ax + By$. Take some $x_0$ and $y_0$ in $\mathbb{R}$, where $x_0 \neq 0$. The eigenvalues of of the skew-symmetric matrix $Ax_0 + By_0$ will be purely imaginary and will exist in conjugate pairs. Note that a given pair may appear more than once. Denote such pairs by $\pm i \lambda_{k}$, where $\lambda_{k}$ is real and $k$ ranges from 1 to $n$.

Then the eigenvalues of $(Ax_0 + By_0)^{2}$ will be $-\lambda^{2}_{k}$, obviously coming in pairs. Since the same pair of eigenvalues of $Ax_0 + By_0$ may occur several times, we cannot say for sure that each $-\lambda^{2}_{k}$ has multiplicity 2, but we can be sure that it has even multiplicity. Let  $v_{k}$ be an eigenvector of $Ax_0 + By_0$ with eigenvalue $i \lambda_{k}$, and $w_{k}$ be an eigenvector of $Ax_0 + By_0$ with eigenvalue $-i \lambda_{k}$. Since $v_{k}$ and $w_{k}$ belong to different eigenspaces of $Ax_0 + By_0$, the subspace $\mbox{span} \lbrace v, w \rbrace$ which they generate is two-dimensional.

 Then $(Ax_0 + By_0)^{2} v_{k} = -\lambda^{2}_{k} v_{k} $ and $(Ax_0 + By_0)^{2} w_{k} = -\lambda^{2}_{k} w_{k} $. But $A^{2} = H$, $AB + BA = K$, and $B^{2} = -I_{2n}$, so we have $$ (Ax_0 + By_{0} )^{2} = Hx_{0}^{2} + Kx_{0}y_{0} -  I_{2n}y_{0}^{2}, $$ so $$(Hx_{0}^{2} + Kx_{0}y_{0}) v_{k} =(y_{0} -\lambda^{2}_{k}) v_{k}. $$ Dividing through by $x_{0} \neq 0$ we have that $$(Hx_{0} + Ky_{0}) v_{k} =\frac{1}{x_{0}}(y_{0} -\lambda^{2}_{k}) v_{k}. $$  Therefore, $v_{k}$ is an eigenvector of $Hx_{0} + Ky_{0}$  with eigenvector $\frac{1}{x_{0}}(y_{0} -\lambda^{2}_{k})$. Repeating this process for $w_{k}$, we see that 
$w_{k}$ is also an eigenvector of $Hx_{0} + Ky_{0}$  with eigenvector $\frac{1}{x_{0}}(y_{0} -\lambda^{2}_{k})$. Therefore $v_k$ and $w_k$ span a two-dimensional eigenspace of $Hx_{0} + Ky_{0}$.

For the case where $x_{0} = 0$, we simply have $K y_{0}$, which clearly has paired eigenvalues because of the diagonal structure of $K$. 
Therefore, for every $x_{0}$ and $y_{0}$ in $\mathbb{R}$ $Hx_0 + Ky_0$ has eigenvalues all of even multiplicity. 
\end{proof}

Denote by $\mathcal{A}$ the algebra generated by $H$ and $K$. The Kippenhahn Conjecture claims that $\mathcal{A}$ cannot be the full algebra $M_{2n}(\mathbb{C})$. We will show that it is in fact the full algebra. 

\subsection{The algebra generated by $H$ and $K$}\label{algebra_generated}

We will show that $H$ and $K$ satisfy the requirements of Theorem \ref{2_gens_thm}. $H$ is clearly symmetric, and $p = 2$. We will first evaluate the characterstics of the 2-blocks of $H$, and draw the associated Burnside graph. Let us evaluate $H_{12}$: $$H_{12} = U \alpha + 2 \alpha U - 2 \alpha^{2} - \beta \alpha. $$ Applying the definitions of $\alpha, \beta \mbox{ and } U$, we see that $$H_{12} =  \alpha U -\beta \alpha - 2 I_{2}. $$ Evaluating this gives us $$H_{12} = \left( \begin{array}{cc}
-2 & -1+b \\
-1-b & -2 
\end{array} \right),$$  which has determinant $3 + b^{2} \neq 0$. In the Burnside graph $B(H, K)$ there must be a connection between nodes 1 and 3, nodes 2 and 4, and either nodes 1 and 4 or nodes 2 and 3, or both. Recall that because $H$ is symmetric, all connections are bi-directional: \begin{center}
\includegraphics[scale=.30]{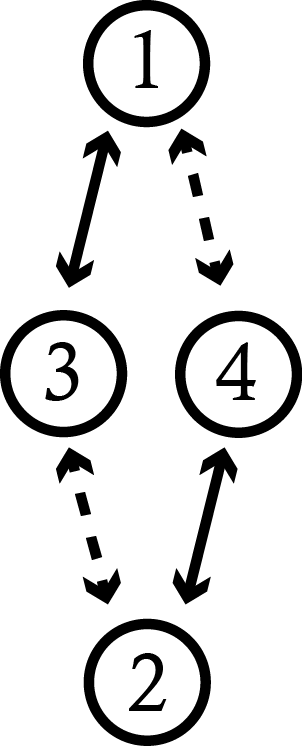}
\end{center} Here the dashed edges indicate that at least one of them has to exist, potentially both.

Now evaluate $H_{13}$: $$H_{13} = U \alpha + \alpha^{2} + 3  \alpha U, $$ which becomes $$H_{13} = 2 \alpha U + I_{2} = \left( \begin{array}{cc}
1 & -2 \\
-2 & 1 
\end{array} \right) = H_{31}, $$ determinant equal to -3. Add the extra connections to the Burnside graph: \begin{center}
\includegraphics[scale=.30]{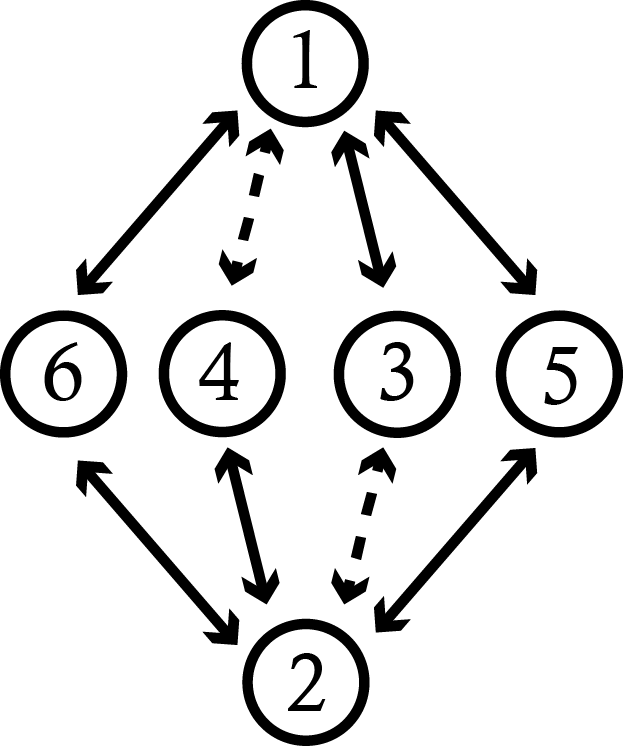}
\end{center} Likewise $$ H_{14} = 3 (\alpha +\beta) U + I_{2} = \left( \begin{array}{cc}
1+3b & -3 \\
-3 & 1-3b 
\end{array} \right) = H_{41},  $$ with determinant $-8 - 9b^{2} \neq 0 $ and for $j = 5,...,n$, $$H_{1j} = (j-1) \alpha U = \left( \begin{array}{cc}
0 & 1-j \\
1-j & 0 
\end{array} \right) = H_{j1}, $$ all of which have non-zero determinant. Add these new nodes and edges, and we see that the Burnside graph is strongly connected:  \begin{center}
\includegraphics[scale=.30]{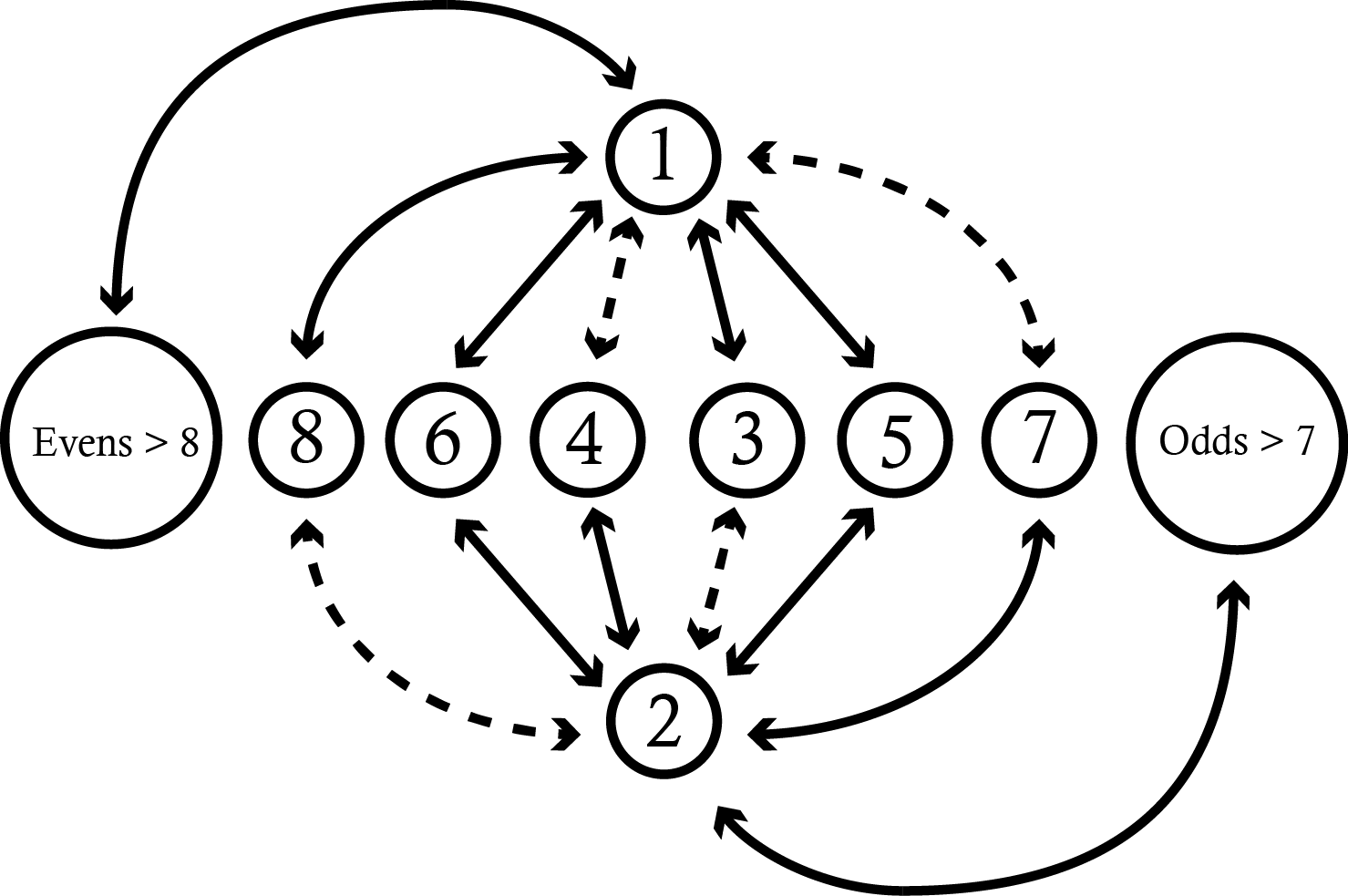}
\end{center} Now we will check off the requirements of Theorem \ref{2_gens_thm} one by one. \begin{enumerate}
\item $K$ clearly satisfies condition Mult$_{2}$.
\item We  have already established that every $H_{1j}$ is invertible. $H$ therefore satisfies Condition $L-2$ via the partition $(1, 1, ... , 1) $ of $n$.
\item 
Consider the  single element 2-words $H_{13}$ and $H_{14}$. Then,
$$H_{13}H_{31} = \left( \begin{array}{cc}
5 & -4 \\
-4 & 5 
\end{array} \right), $$ and likewise $$H_{14}H_{41} = \left( \begin{array}{cc}
9 - (1+3b)^{2} & -6 \\
-6 & 9 + (1+3b)^{2} 
\end{array} \right) $$  Now directly evaluate their commutator. After simplifying, we have: $$[H_{13}H_{31}, H_{14}H_{41}] = \left( \begin{array}{cc}
0 & 48b \\
-48b & 0 
\end{array} \right) \neq 0.  $$ Therefore, $H_{13}$ and $H_{14}$ satisfy the third requirement of Theorem \ref{2_gens_thm}.
\end{enumerate} The pair of matrices $H$ and $K$ therefore satsify the requirements of Theorem \ref{2_gens_thm}, and so $\mathcal{A} = M_{2n}(\mathbb{C})$. Thus $H$ and $K$ as defined are a one-parameter family of counterexamples to Kippenhahn's Conjecture, for order $8 \times 8$ and greater.

\end{document}